\documentclass[a4paper,11pt]{article}

\usepackage[all]{xy}
\usepackage[english]{babel}
\usepackage[latin1]{inputenc}
\usepackage{amsfonts}
\usepackage{amsthm}
\usepackage{amsmath}
\usepackage{amsfonts}
\usepackage{latexsym}
\usepackage{amssymb}
\usepackage{mathrsfs}
\usepackage[usenames]{color}

\newtheorem{fed}{Definition}[section]
\newtheorem{teo}[fed]{Theorem}
\newtheorem{cor}[fed]{Corollary}
\newtheorem*{teo*}{Theorem}

\newtheorem{lem}[fed]{Lemma}
\newtheorem{pro}[fed]{Proposition}

\newtheorem{defi}[fed]{Definition}
\theoremstyle{definition}
\newtheorem{rem}[fed]{Remark}

\newtheorem*{teoD}{Douglas' theorem}

\def\bdem{\begin{proof}}
\def\edem{\renewcommand{\qed}{\hfill $\blacksquare$}
\end{proof}}

\date{}
\begin{document}

\title{Moore-Penrose inverse and partial orders on Hilbert space operators}
\author{Guillermina Fongi $^{a}$ $^*$, M.Celeste Gonzalez $^{b}$ $^{c}$ \footnote{The authors were supported in part by PICT 2017-0883 (FONCYT)}\\ 
\fontsize {9}{9} \selectfont$^a$ 
Centro Franco Argentino de Ciencias de la Informaci\'on y de Sistemas, CIFASIS-CONICET \\ \selectfont \fontsize {9}{9} \selectfont Ocampo y Esmeralda (2000)  Rosario, Argentina.
\\
\fontsize {9}{9} \selectfont$^b$ Instituto Argentino de Matem\'atica ``Alberto P. Calderón'', IAM-CONICET\\
\fontsize {9}{9} \selectfont Saavedra 15, Piso 3 (1083), Buenos Aires, Argentina.\\
\fontsize {9}{9} \selectfont $^c$ Instituto de Ciencias, Universidad Nacional de General Sarmiento.\\
\fontsize {9}{9} \selectfont{$^a$  gfongi@conicet.gov.ar, $^b$ celeste.gonzalez@conicet.gov.ar }}
\date{}
\maketitle


\fontsize {10}{10} \selectfont {{\sl {Keywords: \ }}{Moore-Penrose inverse, operator orders, reverse order law}

\begin{abstract}
In this article we explore several aspects concerning to the Moore-Penrose inverse of a bounded linear operator. On the one hand, we study monotonicity properties of the Moore-Penrose inverse  with respect to the L\"owner, star, minus, sharp and diamond orders. On the other hand, we analyze the validity of the reverse order law, $B^\dagger A^\dagger=(AB)^\dagger$, under hypothesis of operator ranges and also under hypothesis of order operators. 
Finally, we study  the operator $B^\dagger A^\dagger$ as different weighted inverses of $AB$.
\end{abstract}

\maketitle

\section{Introduction}

In this article we study some aspects of the Moore-Penrose inverse ($^\dagger$) of a bounded linear operator on a Hilbert space  by means of different order relations. On the one hand, we analyze the behavior of the  Moore Penrose inverse for the L\"owner, star, minus, sharp and diamond orders. On the other hand we study the operator $B^\dagger A^\dagger$ as a generalized inverse of $AB$ in two different ways: firstly we study the reverse order law for the Moore Penrose inverse, $B^\dagger A^\dagger=(AB)^\dagger$, under conditions of orders between the operators $A$ and $B$; secondly we study the operator $B^\dagger A^\dagger$ as diferent weighted inverses of $AB$.

The Moore-Penrose inverse of a matrix  was defined by Moore \cite{moore1920reciprocal} in 1920 and independently, by Penrose \cite{MR69793} in 1955.  Later, this concept was extended to the context of infinite dimensional Hilbert space operators and since then, it has been extensively studied due to its numerous applications.   

Regarding the topics covered in this work,  a treatment on the behavior of the Moore Penrose inverse of a matrix with respect  to the L\"owner order can be found in \cite{MR1048801}. There, it is proved  that  the Moore Penrose inverse is decreasing on certain subsets of Hermitian matrices.  In this article, we study the monotonicity of the Moore-Penrose inverse with respect to the L\"owner order on the cone of positive operators defined on a Hilbert space. In addition, we show that if we consider the L\"owner and the star orders on the set of Hermitian operators  then the Moore Penrose inverse is monotonic increasing  with respect to these two orders; in this case we also show that the star order can not be changed by the minus or diamond order. In \cite{arias2021partial} it is shown that the Moore-Penrose inverse is  monotonic increasing with respect to the star order on the set of Hilbert space operators. Here, we prove that the monotonicity of the Moore-Penrose inverse holds for the minus and diamond orders,  under a condition that we call  \textit{range dagger sustractivity property}. We  also show that these results are not necessarily valid without this additional condition. 
 
 The reverse order law for the Moore Penrose inverse has been studied for a long time  and several characterizations for its validity (in the finite and infinite dimensional context) have been obtained in terms of range and factorization conditions, see for example  \cite{greville1966note, MR326424, MR2567299,MR651705}.  We present a sufficient condition for EP operators that allows to guarantee the reverse order law for the infinite dimensional case. We also show that this sufficient condition is weaker than one known in the matricial case. In the matricial context the reverse order law  was also studied under conditions of  space pre-order, star and minus orders, see for example  \cite{MR2653538,malik2007matrix}. We extend  certain results of these articles to the context of infinite dimensional Hilbert space operators. In particular, we prove that the reverse order law can be guaranteed under a star order condition. Also, we show that the star order hypothesis   can not be relaxed by a minus order condition.

We  also study the case that $B^\dagger A^\dagger$  is not necessarily the Moore-Penrose inverse of $AB$. 
In \cite{MR2187919, MR3085500}  different types of weighted inverses of an operator are  studied. In this article we study  when the operator $B^\dagger A^\dagger$ is a weighted inverse of $AB$ for certain particular weights. If the reverse order law holds  then $B^\dagger A^\dagger$ is a solution of the equation $ABX=P_{\mathcal{R}(A)}$ if and only if $\mathcal{R}(A)=\mathcal{R}(AB)$. Otherwise (i.e. if $B^\dagger A^\dagger$ is not necessarily the Moore-Penrose inverse of $AB$), it is interesting to analyze when $B^\dagger A^\dagger$ is a  best approximation solution (in some sense) of the problem $ABX-P_{\mathcal{R}(A)}$. In \cite{ConGirMae17} the following  problem is solved: given $T,S\in\mathcal{L}(\mathcal{H})$, analyze the existence of a minimum of the set $\{\|W^{1/2}(TX-S)\|_p: X\in\mathcal{L}(\mathcal{H})\}$ and its relationship with weighted inverses of $T$,  where $\|\cdot\|_p$ is a $p-$Schatten norm and $W$ a positive operator of $\mathcal{L}(\mathcal{H})$.  Inspired in this treatment, in the case the reverse order law does not necessarily hold we study when the operator $B^\dagger A^\dagger$ is the best solution of the problem $\|W^{1/2}(ABX-P_{\mathcal{R}(A)})\|_p$ for certain  positive weight $W$.

The contents of the article are the following: in Section \ref{Preliminares} we introduce  notations,  definitions and results that we will be used along the paper. In Section \ref{seccion-ordenes} we introduce the definitions of the minus, star, diamond and sharp orders and  certain relationships between them. In this section, our main contributions are Theorem \ref{minus equiv diamond} and Theorem \ref{minus y DSP entonces estrella}. In the first one,  we prove that the minus and the diamond orders are equivalent under the  range dagger sustractivity property.  In the second one, we show that the star order is equivalent to the the minus order if the dagger sustractivity property holds.
	
In Section \ref{seccion-monoticidad} we analyze the monotonicity of the Moore-Penrose map for the different classes of orders. Here the main results are: Theorems \ref{lowner-order} and  \ref{lowner y estrella}, where the monotonicity of the Moore Penrose inverse is studied with respect to  the star and the L\"owner  orders. Also in Proposition \ref{minusydagger} and Corollary \ref{diamondydagger}, the monotonicity is analyzed with respect to the minus and diamond orders under the range dagger sustractivity property. Finally, in Section \ref{seccion-ROL} we study  the operator $B^\dagger A^\dagger$ as a generalized inverse of $AB$. In the first part of this section we study the validity of the equality $(AB)^\dagger=B^\dagger A^\dagger$ under different conditions of range and order between the operators $A$ and $B$.  In Theorem \ref{AEP-reverseorderlaw}  we stablish a sufficient condition to get the reverse order law in terms of EP operators.   In Propositions \ref{ROL-AEP} and \ref{ROL-estrella} we show that the reverse order law is valid under a star order condition.

 The last part of this section is devoted to study  the operator $B^\dagger A^\dagger$ as a weighted generalized inverse of $AB$  and also  to study $B^\dagger A^\dagger$ as a solution of a  weighted operator least squares problem related to $AB$, when the weighted norm comes from a Schatten $p$-norm. The main results on these subjects are Theorems \ref{BdaggerAdagger-wgi} and \ref{BdaggerAdagger-pwproblem}.

\section{Preliminaries}\label{Preliminares}

Along the article $\mathcal{H}$ denotes a complex Hilbert space and $\mathcal{L}(\mathcal{H})$ is the algebra of bounded linear operators from $\mathcal{H}$ to $\mathcal{H}$. If $T\in\mathcal{L}(\mathcal{H})$ then $T^*$ denotes the adjoint of $T$, $\mathcal{R}(T)$ is the range of $T$ and $\mathcal{N}(T)$ is the nullspace of $T$. The subsets of  $\mathcal{L}(\mathcal{H})$ of Hermitian operators, positive operators, idempotent operators and orthogonal projections operators are denoted by $\mathcal{L}^h$, $\mathcal{L}^+$,  $\mathcal{Q}$ and $\mathcal{P}$, respectively. Given two subspaces $\mathcal{S}, \mathcal{T}$ of $\mathcal{H}$ then $\mathcal{S}\dot+\mathcal{T}$ and $\mathcal{S}\oplus\mathcal{T}$ denote the direct and orthogonal sums between $\mathcal{S}$ and $\mathcal{T}$, respectively. If $\mathcal{S}$, $\mathcal{T}$ are closed subspaces such that $\mathcal{H}=\mathcal{S}\dot+\mathcal{T}$, then the idempotent operator with range $\mathcal{S}$ and nullspace $\mathcal{T}$ is denoted by $Q_{\mathcal{S}//\mathcal{T}}$.  The orthogonal projection onto $\overline{\mathcal{R}(T)}$ is denoted by $P_T$.

Given $T\in\mathcal{L}(\mathcal{H})$, the Moore-Penrose inverse of $T$ is $T^\dagger: \mathcal{R}(T)\oplus\mathcal{R}(T)^\bot\longrightarrow \mathcal{H}$ such that $T^\dagger|_{\mathcal{R}(T)}=(T|_{\mathcal{N}(T)^\bot})^{-1}$ and $\mathcal{N}(T^\dagger)=\mathcal{R}(T)^\bot$, where $T|_{\mathcal{N}(T)^\bot}: \mathcal{N}(T)^{\bot}\longrightarrow \mathcal{R}(T)$. The Moore-Penrose inverse is in general a densely defined unbounded operator. It is well known that $T^\dagger\in\mathcal{L}(\mathcal{H})$ if and only if $\mathcal{R}(T)$ is closed; in such case,  $T^\dagger$ is characterized as the unique solution of the following four equations

\begin{equation}\label{moore-penrose-equations}
1. \ TXT=T,  \ \ \  2. \ XTX=X, \ \ \ 3. \ (TX)^*=TX, \ \ \  4. \ (XT)^*=XT.
\end{equation}

Equations (\ref{moore-penrose-equations}) characterize different kinds of pseudoinverses for  $T\in \mathcal{L}(\mathcal{H})$ with closed range. More precisely, $T'\in\mathcal{L}(\mathcal{H})$ is an inner inverse of $T$ if and only if it satisfies equation 1.,  $T'\in\mathcal{L}(\mathcal{H})$ is an outer inverse of $T$ if  and only if it satisfies equation 2., $T'\in\mathcal{L}(\mathcal{H})$ is a generalized inverse of $T$ if and only if it satisfies equations 1 and 2. For simplicity, we denote by $T[h,i]$, $T[h,i,,j]$, $T[h,i,j,k]$  the subsets of operators in $\mathcal{L}(\mathcal{H})$ which satisfy equations $h,i$; $h,i,j$ and $h,i,j,k$ for the values $h,i,j,k :=1,2,3,4$. Observe that, $\{T^\dagger\}=T[1,2,3,4]$.

The equations in (\ref{moore-penrose-equations}) were generalized in order to define different classes of pseudoinverses. The concept of generalized inverses with positive  weights  was introduced by Ben-Israel and Greville \cite{MR587113} in the context of finite dimensional Hilbert spaces.  Later it was  studied by Corach and Maestripieri \cite{MR2187919}  in infinite dimensional Hilbert spaces by means of the theory of compatibility.

\begin{defi}\label{MN-weighted}
Let $M,N\in \mathcal{L}^+$. Given $T\in \mathcal{L}(\mathcal{H})$ with closed range,  an operator $T' \in \mathcal{L}(\mathcal{H})$ is called an $M,N$-weighted generalized inverse of $T$ if $T'$ satisfies the following four equations:
$$
 \ TXT=T,  \ \ \  \ XTX=X, \ \ \  \ (MTX)^*=MTX, \ \ \   \ (NXT)^*=NXT.
$$
\end{defi}

The following concept was introduced by Rao and Mitra for finite dimensional spaces \cite{MR352148} and later it was extended by Corach, Fongi and Maestripieri to infinite dimensional Hilbert spaces \cite{MR3085500}.

\begin{defi}
	Let $M \in\mathcal{L}^+$. Given $T\in \mathcal{L}(\mathcal{H})$ with closed range, an operator $T'\in\mathcal{L}(\mathcal{H})$ is called an  $M$-inverse of $T$ if for each $y\in\mathcal{H}$, $T'y$ is an M-least square solution (M-LSS) of $Tx=y$, i.e.,
		$$
		\|y-BT'y\|_M\leq \|y-Tx\|_M,  \ \ x\in \mathcal{H}.
		$$
		\end{defi}

\begin{pro}\label{caracterizacion de M-inversas}\cite[Proposition 5.9 ]{MR3085500}
Consider a closed range operator $T\in \mathcal L(\mathcal H)$.	An operator $T'\in \mathcal L(\mathcal H)$ is an $M$-inverse of $T$ if and only if $T^*MTT'=T^*M$.
	\end{pro}

On the other hand, given $T\in \mathcal{L}({\mathcal{H}})$ with closed range, a generalized inverse $T'$ of $T$ which satisfies the equation $TX = X T$ is called the \textit{group inverse} of $T$ and it is denoted by $T^\sharp$. It holds that if this operator exists then it is unique.  In \cite{MR229658},  it is shown necessary and sufficient conditions for the existence of this pseudoinverse:
\begin{teo}\label{group inverse conditions}
	Consider  $T\in \mathcal{L}(\mathcal{H})$, then the following statements are equivalent:
	\begin{enumerate}
		\item  $\mathcal H=\mathcal{R}(T)\dotplus \mathcal{N}(T)$,
		\item $\mathcal{R}(T^2)=\mathcal{R}(T)$, $\mathcal{N}(T^2)=\mathcal{N}(T)$,
		\item $T^\sharp $ exists.
	\end{enumerate}
\end{teo}
If $T$ admits a group inverse then it is well known that $T$ has  closed range, $\mathcal{R}(T^\sharp)=\mathcal{R}(T)$, $\mathcal{N}(T^\sharp)=\mathcal{N}(T)$ and $TT^\sharp=T^\sharp T=Q_{\mathcal{R}(T)//\mathcal{N}(T)}$. In addition, $T$ admits a group inverse if and only if $T^*$ admits a group inverse. On the other hand, if $T \in \mathcal{L}(\mathcal{\mathcal{H}})$  is an EP operator (i.e., $T$ has closed range and  $\mathcal{R}(T)=\mathcal{R}(T^*)$)  then it admits a group inverse.

The following result on range inclusion and factorization is due to Douglas \cite{MR0203464}:

\begin{teoD}
	Let $A,B \in \mathcal{L}(\mathcal{H})$. The following conditions are equivalent:
	\begin{enumerate}
		\item $\mathcal{R}(B)\subseteq \mathcal{R}(A)$;
		\item there exists a number $\lambda >0$ such that $BB^*\leq \lambda AA^*$;
		\item there exists $C\in\mathcal{L}(\mathcal{H})$ such that $AC=B$.
	\end{enumerate}
	In addition, if any of the above conditions holds then there exists a unique operator $X_r\in\mathcal{L}(\mathcal{H})$ such that $AX_r=B$ and $\mathcal{R}(X_r)\subseteq \mathcal{N}(A)^\bot$. Futhermore, $\|X_r\|=inf \{\lambda: \  BB^*\leq \lambda AA^*\}$. The operator $X_r$ is called the \it{Douglas' reduced solution} of $AX=B$.
\end{teoD}

\section{On operator orders}\label{seccion-ordenes}

We start this section by defining the operator orders that will be covered in this article. Also  some relationships between them will be analyzed. The space pre-order between operators in $\mathcal{L}(\mathcal{H})$ is denoted by $\overset{s}{\leq}$. The classical L\"owner order for operators in $\mathcal{L}^h$  is indicated by $\leq$.   Finally, 
the star, minus, diamond and sharp orders defined on $\mathcal{L}(\mathcal{H})$ are denoted by $\overset{*}\leq$, $\overset{-}\leq$, $\overset{\diamond}\leq$ and $\overset{\sharp}\leq$, respectively.

\begin{defi}\label{ordenes}
Let $A,B\in \mathcal{L}(\mathcal{H})$. Then:
\begin{enumerate}
\item $A\overset{s}\leq B$  if $\mathcal{R}(A)\subseteq \mathcal{R}(B)$ and $\mathcal{R}(A^*)\subseteq \mathcal{R}(B^*)$.
\item $A\overset{*}\leq B$ if $A^*A=A^*B$ and $AA^*=BA^*$.
\item $A\overset{\diamond}\leq B$ if $A\overset{s}\leq B$ and $AA^*A=AB^*A$.
\item $A\overset{-}\leq B$ if  there exist $P,Q\in \mathcal{Q}$ such that $A=PB$ and $A^*=QB^*$. The ranges of $P$ and $Q$ can be fixed as $\mathcal R(P)=\overline{\mathcal R(A)}$ and $\mathcal R(Q)=\overline{\mathcal R(A^*)}$.

\item If $A, B$ admit group inverses then,  $A\overset{\sharp}\leq B$ if  there exists $Q\in\mathcal{Q}$ with $\mathcal{R}(Q)=\mathcal{R}(A)$, $\mathcal{N}(Q)=\mathcal{N}(A)$ such that $A=QB$ and $A=BQ$.
\end{enumerate}
\end{defi}

We refer  the reader to \cite{Hart,Namboo, MR2742334,DC,  jose2015partial, arias2021partial,LT,MR3682701} for different treatments of these type of orders.

\begin{pro}\label{rem-orden-estrella} If $A,B\in \mathcal{L}(\mathcal{H})$  then the following characterizations hold: 
\begin{enumerate} 
\item\label{estrellaproyecciones} $A\overset{*}\leq B$ if and only if  there exist $P,Q\in \mathcal{P}$ such that $A=PB$ and $A^*=QB^*$. The orthogonal projections can be choosen such that $P=P_A$ and $Q=P_{A^*}$.
\item\label{estrellaMP} If  $A,B$ have closed ranges then $A\overset{*}\leq B$ if and only if $AA^\dagger=BA^\dagger=AB^\dagger$ and $A^\dagger A=A^\dagger B=B^\dagger A$. 
\item\label{rem-sharp}  If  $A,B$ are group invertible then $A\overset{\sharp}\leq B$ if and only if $A^2 =BA = AB$. \label{djsh}
\end{enumerate}
\end{pro}

Characterization  \ref{estrellaproyecciones} of the above proposition can be found in \cite[Proposition 2.3]{MR2742334}. Characterizations \ref{estrellaMP}  and \ref{rem-sharp} are stated in \cite[ Remark 5.5 and Theorem 3.4]{jose2015partial}, respectively.

The minus, star, diamond and sharp orders can be also characterized by means of  operator range decompositions. The proof of items $1,2$ and $3$ of the following proposition can be found in  \cite[Theorem 3.3 and Corollary 3.9]{MR3682701} and \cite[Corollary 5.4]{arias2021partial}. The characterization given in item $4$ is new.

\begin{pro}\label{caracterizaciones ordenes}
Consider $A,B\in \mathcal{L}{(\mathcal{H})}$. Then,
\begin{enumerate}
	\item $A\overset{-}\leq B$ if and only if  $\mathcal{R}(B)=\mathcal{R}(A) \overset{.}{+}\mathcal{R}(B-A)$ and $\mathcal{R}(B^*)=\mathcal{R}(A^*) \overset{.}{+}\mathcal{R}(B^*-A^*)$.
	\item $A\overset{*}\leq B$ if and only if  $\mathcal{R}(B)=\mathcal{R}(A) \oplus\mathcal{R}(B-A)$ and $\mathcal{R}(B^*)=\mathcal{R}(A^*) \oplus\mathcal{R}(B^*-A^*)$.
	\item If $A, B$ have closed ranges then   $A\overset{\diamond}\leq B$ if and only if  $\mathcal{R}(B)=\mathcal{R}(A) \dot +\mathcal{R}((B^\dagger-A^\dagger)^*)$ and $\mathcal{R}(B^*)=\mathcal{R}(A^*) \dot +\mathcal{R}(B^\dagger-A^\dagger)$.
	\item If  $A,B$ are group invertible then $A\overset{\sharp}\leq B$  if and only if $\mathcal{R}(B)=\mathcal{R}(A) \dot +\mathcal{R}(B-A)$ where  $\mathcal{R}(B-A)\subseteq \mathcal{N}(A)$ and $\mathcal{R}(A)\subseteq \mathcal {N}(B-A)$.
\end{enumerate}
\end{pro}

\begin{proof}
We only prove  item 4.  Suppose that $A\overset{\sharp}\leq B$. Then,  $A=QB=BQ$, where $Q=Q_{\mathcal{R}(A)//\mathcal{N}(A)} \in \mathcal Q$. Note that, in this case, $\mathcal{R}(A)\subseteq\mathcal{R}(B)$. Thus $\mathcal{R}(B)=\mathcal{R}(A)+\mathcal{R}(B-A)$. Now, since $A=QB$ then $B-A=(I-Q)B$, so that  $\mathcal{R}(B-A)\subseteq \mathcal{N}(A)$ and then $\mathcal{R}(A)\cap\mathcal{R}(B-A)=\{0\}$. Also,  $\mathcal{R}(A)\subseteq \mathcal{N}(B-A)$ because $B-A=B(I-Q)$. Conversely, since $A$ admits a group inverse then there exists $Q=Q_{\mathcal{R}(A)//\mathcal{N}(A)}\in \mathcal Q$. Then $Q(B-A)=0$ and $(B-A)Q=0$, because $\mathcal{R}(B-A)\subseteq \mathcal{N}(A)$ and $\mathcal{R}(A)\subseteq \mathcal {N}(B-A)$. Then $A=QB=BQ$, so that $A\overset{\sharp}\leq B$.   
\end{proof}

In what follows we study certain relationships between the mentioned partial orders. The next result collects some known facts.
\begin{pro}\label{relationships between orders} 
Let  $A,B\in \mathcal{L}{(\mathcal{H})}$. Then, the following assertions hold:
\begin{enumerate}
\item $A\overset{*}\leq B $ if and only if $ A\overset{-}\leq B$ and $AB^*, B^*A \in\mathcal{L}^h$.
\item $A\overset{*}\leq B $ if and only if $ A\overset{\diamond}\leq B$ and  $AB^*, B^*A\in\mathcal{L}^h$. 
\item If $A,B$ admit group inverses then $A\overset{\sharp}\leq B $ if and only if $ A\overset{-}\leq B$ and $AB=BA$.
\item\label{star-sharp} If $A$ is an EP operator  and $B$ is group invertible then $A\overset{*}\leq B $ if and only if $ A\overset{\sharp}\leq B$. 
\end{enumerate}
\end{pro}

\begin{proof}
1.  It follows from  \cite[Theorem 5.10]{jose2015partial} since this result can be extended with the same proof for non necessarily closed range operators.

2. Suppose that  $ A\overset{\diamond}\leq B$ and $AB^*, B^*A\in\mathcal{L}^h$. Then $AA^*A=AB^*A$, and pre-multiplying by $A^\dagger$, we get that $A^*A=P_{A^*}B^*A=P_{A^*}A^*B=A^*B$. Similarly, it holds that $AA^*=AB^*$. Hence $ A\overset{*}\leq B$. The converse follows from Definition \ref{ordenes} and Proposition \ref{rem-orden-estrella}.

3. See \cite[Theorem 5.9]{jose2015partial}.

4. First note that $A$ is group invertible, by Proposition \ref{group inverse conditions}. Then the equivalence follows from Definition \ref{ordenes} and item \ref{estrellaproyecciones} of Proposition \ref{rem-orden-estrella}.
\end{proof}

\begin{cor}\label{equiv1,estrella, menos,diamante}
Consider $A,B\in  \mathcal{L}{(\mathcal{H})}$ such that $AB^*, B^*A\in\mathcal{L}^h$. Then $A\overset{*}\leq B $ if and only if $ A\overset{\diamond}\leq B$ if and only if  $ A\overset{-}\leq B$.
\end{cor}

In the following result we present another condition under which the minus and diamond orders are equivalent. 

\begin{teo}\label{minus equiv diamond}
Let  $A, B\in\mathcal{L}(\mathcal{H})$ with closed ranges such that  $\mathcal{R}((B-A)^\dagger)=\mathcal{R}(B^\dagger-A^\dagger)$ and $\mathcal{R}((B^*-A^*)^\dagger)=\mathcal{R}((B^*)^\dagger-(A^*)^\dagger)$. Then
 $A\overset{-}\leq B$ 
	if and only if $A\overset{\diamond}\leq B$.
\end{teo}

\begin{proof}
 Suppose   $A\overset{-}\leq B$, then $\mathcal{R}(B)=\mathcal{R}(A) \overset{.}{+}\mathcal{R}(B-A)$ and $\mathcal{R}(B^*)=\mathcal{R}(A^*) \overset{.}{+}\mathcal{R}(B^*-A^*)$. Since $A$ and $B$ have closed range then $R(B-A)$ is closed, see  \cite[Theorem 2.3]{fillmore1971operator}. 
	Now, observe that
	\begin{eqnarray}
	\mathcal{R}(B)&=&\mathcal{R}(A) \overset{.}{+}\mathcal{R}(B-A)=\mathcal{R}(A) \overset{.}{+}\mathcal{N}((B-A)^*)^\bot = \mathcal{R}(A) \overset{.}{+}\mathcal{R}(((B-A)^*)^\dagger) \nonumber\\
	&=&\mathcal{R}(A) \overset{.}{+}\mathcal{R}((B^\dagger-A^\dagger)^*),\nonumber
	\end{eqnarray}
	and similarly, $ \mathcal{R}(B^*)=\mathcal{R}(A^*) \overset{.}{+}\mathcal{R}(B^\dagger-A^\dagger)$.
	Therefore, by Proposition \ref{caracterizaciones ordenes} it holds that $A\overset{\diamond}\leq B$. Conversely, if $A\overset{\diamond}\leq B$ then 
	\begin{eqnarray}
	\mathcal{R}(B)&=&\mathcal{R}(A) \overset{.}{+}\mathcal{R}((B^*)^\dagger-(A^*)^\dagger)=\mathcal{R}(A) \overset{.}{+}\mathcal{R}((B^*-A^*)^\dagger) = \mathcal{R}(A) \overset{.}{+}\mathcal{N}((B-A)^*)^\bot \nonumber\\
	&=&\mathcal{R}(A) \overset{.}{+}\overline{\mathcal{R}(B-A)},\nonumber
	\end{eqnarray}
	and similarly, $ \mathcal{R}(B^*)=\mathcal{R}(A^*) \overset{.}{+}\overline{\mathcal{R}(B^*-A^*)}.$
Then,  $A\overset{-}\leq B$, see \cite[Proposition 3.2]{MR3682701}
\end{proof}

\medskip

	If $A, B\in \mathcal{L}(\mathcal{H})$ are such that $B$ has closed  range and  $A\overset{*}\leq B$ then  it holds that
	$$
	A\overset{-}\leq B \textrm{ and } (B-A)^\dagger=B^\dagger-A^\dagger,
	$$
	see \cite[Theorem 6.1]{jose2015partial} and \cite[Proposition 3.6 and Corollary 4.10]{MR3682701}. The next result proves that converse of the above statement  holds.	
	\begin{teo}\label{minus y DSP entonces estrella}Consider  $A, B\in \mathcal{L}(\mathcal{H})$  such that $A$ and $B$ have closed  range. Then  $A\overset{*}\leq B$ then  if and only if
	$
	A\overset{-}\leq B \textrm{ and } (B-A)^\dagger=B^\dagger-A^\dagger,
	$
\end{teo}

\begin{proof}
		Suppose $A\overset{-}\leq B$ and $(B-A)^\dagger=B^\dagger-A^\dagger$. Since $\mathcal R(B)$ is closed and $A\overset{-}\leq B$, then by \cite[Corollary 3.16 ]{MR3682701} it follows that  $\mathcal R(B-A) $ are also closed.  Since  $\mathcal R(B)=\mathcal R(A)\dot{+} \mathcal R(B-A)$ then it holds that $\mathcal R(A)^\perp \cap \mathcal{R}(B)\neq \{0\}$. In fact, suppose that  $\mathcal R(A)^\perp \cap \mathcal{R}(B)=\{0\}$. Then $\mathcal{R}(A)\dot+\mathcal{R}(B)^\bot=\mathcal{H}$ because  $\mathcal{R}(A)\dot+\mathcal{R}(B)^\bot$ is dense and closed. In addition, $\mathcal{R}(B)=\mathcal{R}(A)$. In fact, it is clear that $\mathcal{R}(A)\subseteq \mathcal{R}(B)$. Now, let $Bx\in\mathcal{R}(B)$, for some $x\in\mathcal{H}$. Then $Bx=Az+w$, for some $z\in\mathcal{H}$ and $w\in\mathcal{R}(B)^\bot$ and so that $Bx-Az\in\mathcal{R}(B)\cap\mathcal{R}(B)^\bot=\{0\}$. Therefore, $Bx\in\mathcal{R}(A)$ and thus $\mathcal{R}(B)=\mathcal{R}(A)$. Now, since $\mathcal R(B)=\mathcal R(A)\dot{+} \mathcal R(B-A)$ so that $\mathcal{R}(B-A)=\{0\}$. Then $B=A$ which is an absurd. In consequence, $\mathcal R(A)^\perp \cap \mathcal{R}(B)\neq \{0\}$. 
		
		Now, we prove that  $ \mathcal R(A)^\perp \cap \mathcal{R}(B)\subseteq \mathcal R(B-A)$: consider $u\in \mathcal R(A)^\perp \cap \mathcal{R}(B),$ with $u\neq 0$. Then $ u=P_{\mathcal{R}(B)}u=BB^\dagger u=B((B-A)^\dagger+A^\dagger)u=B((B-A)^\dagger u=B((B-A)^\dagger u -A(B-A)^\dagger u+ A(B-A)^\dagger u)= (B-A)(B-A)^\dagger u+ A((B-A)^\dagger u$, where the fourth equality follows because $u\in \mathcal R(A)^\perp=\mathcal N(A^\dagger)$. Therefore $u=P_{\mathcal R(A)^\perp}u=P_{\mathcal R(A)^\perp}P_{\mathcal R(B-A)}u$, so that $u\in \mathcal R(B-A)$. In fact,  first note that $\|u\|\leq \|P_{\mathcal{R}(B-A)} u\|$. Now, since $u=u_1+u_2$, where $u_1\in\mathcal{R}(B-A)$ and $u_2\in\mathcal{R}(B-A)^\bot$. Then $P_{\mathcal{R}(B-A)}u=u_1$. If $u\not\in\mathcal{R}(B-A)$ then $\|P_{R(B-A)u}\|<\|u\|$ which is a contradiction. Therefore, $ \mathcal R(A)^\perp \cap \mathcal{R}(B)\subseteq \mathcal R(B-A)$. 
		
		Finally, we will show that $\mathcal R(B-A)\subseteq \mathcal R(A)^\perp$. Consider $(B-A)y\in \mathcal R(B-A)$. Since $\mathcal S=\mathcal R(A)^\perp \cap \mathcal{R}(B)\subseteq \mathcal R(B-A)$, then $(B-A)y-P_{\mathcal S}(B-A)y=(I-P_{\mathcal S})(B-A)y\in  \mathcal  R(B-A)\cap \mathcal S^\perp \subseteq \mathcal R(B)\cap (\mathcal R(A)^\perp\cap \mathcal R(B))^\perp=\mathcal R(A)$.  Therefore, $(B-A)y-P_{\mathcal S}(B-A)\in  \mathcal R(A)\cap \mathcal R(B-A)=\{0\}$, so that $(B-A)y=P_{\mathcal S}(B-A)y\in  \mathcal R(A)^\perp$. Then, it holds that $\mathcal R(B-A)\subseteq \mathcal R(A)^\perp$.  Hence  $\mathcal R(B)=\mathcal R(A)\oplus \mathcal R(B-A)$. Similarly it can be proved that  $\mathcal R(B^*)=\mathcal R(A^*)\oplus \mathcal R(B^*-A^*)$. Hence, by Proposition \ref{caracterizaciones ordenes}, it holds that $A\overset{*}\leq B$. For the converse see \cite[Theorem 6.1]{jose2015partial} or   \cite[Proposition 3.6 and Corollary 4.10]{MR3682701}.
\end{proof}

\begin{cor}\label{equiv2,estrella, menos,diamante}
Consider $A,B\in  \mathcal{L}{(\mathcal{H})}$ with closed ranges such that  $(B-A)^\dagger=B^\dagger-A^\dagger$. Then $A\overset{*}\leq B $ if and only if $ A\overset{\diamond}\leq B$ if and only if  $ A\overset{-}\leq B$.
\end{cor}

\begin{rem}
It is worth noting that with the same proof as that of  Theorem \ref{minus y DSP entonces estrella}, a more general result follows: consider  $A, B\in \mathcal{L}(\mathcal{H})$ such that $B$ has closed range. Then $A\overset{*}\leq B$ if and only if $A\overset{-}\leq B$ and     $(A+B)(A^\dagger +B^\dagger)=P_{\mathcal{R}(A+B)}$. 
\end{rem}

\begin{rem} \label{relacion entre hipotesis} Recall that under the hypothesis $AB^*, B^*A\in\mathcal{L}^h$ or 
$(B-A)^\dagger=B^\dagger-A^\dagger$ we showed that the star, diamond and minus orders are equivalent (see Corollary \ref{equiv1,estrella, menos,diamante} and Corollary \ref{equiv2,estrella, menos,diamante}). Now, let us see that conditions $AB^*, B^*A\in\mathcal{L}^h$ and 
$(B-A)^\dagger=B^\dagger-A^\dagger$ are not related.
In fact, consider $\mathcal{H}=\mathbb{R}^2$ and $A,B\in\mathcal{L}(\mathcal{H})$ given by  $A=\left(\begin{array}{cc}
1& 0 \\
1 & 0
\end{array}\right)$ and $B=\left(\begin{array}{cc}
1& 1 \\
1 & -1
\end{array}\right)$.  Then $(B-A)^\dagger= \left(\begin{array}{cc}
0& 1 \\
0 & -1
\end{array}\right)^\dagger= \left(\begin{array}{cc}
0& 0 \\
1/2 & -1/2
\end{array}\right)$ and $B^\dagger-A^\dagger= \left(\begin{array}{cc}
1/2& 1/2 \\
1/2 & -1/2
\end{array}\right)- \left(\begin{array}{cc}
1/2& 1/2 \\
0 & 0
\end{array}\right)=(B-A)^\dagger$. However, $B^*A\not\in\mathcal{L}^h$. On the other hand, if $A=\left(\begin{array}{cc}
2& 0 \\
0 & 0
\end{array}\right)$ and $B=\left(\begin{array}{cc}
1& 0 \\
0 & 2
\end{array}\right)$ it holds that $AB^*, B^*A\in\mathcal{L}^h$ but $(B-A)^\dagger= \left(\begin{array}{cc}
-1& 0 \\
0 & 1/2
\end{array}\right)\neq \left(\begin{array}{cc}
1/2& 0 \\
0 & 1/2
\end{array}\right) = B^\dagger - A^\dagger$.
\end{rem}

Given  $A,B\in \mathcal{L}(\mathcal H)$, condition $(B-A)^\dagger=B^\dagger-A^\dagger$  is usually known as \textit{dagger substractivity property}.
We say that $A$ and $B$ satistfy the \textit{range dagger substractivity property}  if 
$$
\mathcal{R}((B-A)^\dagger)=\mathcal{R}(B^\dagger-A^\dagger) \ \textrm{and} \  \mathcal{R}((B^*-A^*)^\dagger)=\mathcal{R}((B^*)^\dagger-(A^*)^\dagger).
$$ Observe that the range dagger sustractivity property does not imply dagger sustractivity property, in general. In fact, let $A=\left(\begin{array}{cc}
2& 0 \\
0 & 0
\end{array}\right)$ and $B=\left(\begin{array}{cc}
1& 0 \\
0 & 2
\end{array}\right)$ in $\mathcal{L}^+$ as in Remark \ref{relacion entre hipotesis}. Then, $\mathcal{R}((B-A)^\dagger)=\mathcal{R}(B^\dagger-A^\dagger)$ but $(B-A)^\dagger \neq B^\dagger - A^\dagger$. More generally, given two diagonal operators $A$ and $B$ in $\mathcal{L}(\mathcal{H})$ (i.e. $Ax=\sum \lambda_jx_je_j$ and $Bx=\sum \mu_jx_je_j$, with $x=\sum x_je_j$ and $\{e_j\}$ an orthonormal basis of $\mathcal{H}$) then $A$ and $B$ satisfy the range dagger sustractivity property but they do not  necessarily satisfy the dagger sustractivity property.

\medskip

We finish this section by giving a characterization of the range dagger sustractivity property.

\begin{pro}
	Consider $A, B\in\mathcal{L}(\mathcal{H})$ with closed ranges such that  $\mathcal{R}(B-A)$ is closed. Then   $\mathcal{R}((B-A)^\dagger)= \mathcal{R}(B^\dagger-A^\dagger)$ and
	$\mathcal{R}((B^*-A^*)^\dagger)=\mathcal{R}((B^*)^\dagger-(A^*)^\dagger)$ if and only if there exist $X,Y\in\mathcal L(\mathcal{H})$ invertible operators such that $$
	(B^\dagger  -A^\dagger) X=B^*-A^*  
	\textrm{ and }
	((B^*)^\dagger  -(A^*)^\dagger) Y=B-A.	$$
\end{pro}

\begin{proof}
	Suppose that $\mathcal{R}(B-A)$ is closed, $\mathcal{R}((B-A)^\dagger)=\mathcal{R}(B^\dagger-A^\dagger)$ and $\mathcal{R}((B^*-A^*)^\dagger)=\mathcal{R}((B^*)^\dagger-(A^*)^\dagger)$. Then $\mathcal{R}(B^*-A^*) = \mathcal N(B-A)^\perp=\mathcal{R}((B-A)^\dagger)=\mathcal{R}(B^\dagger-A^\dagger)$ and  $\mathcal N(B^\dagger-A^\dagger)=\mathcal R((B^\dagger-A^\dagger)^*)^\perp=\mathcal{R}((B^*-A^*)^\dagger)^\perp=\mathcal{N}(B^*-A^*)$. Hence, by \cite[Corollary 1]{fillmore1971operator} there exists an invertible operator $X\in\mathcal L(\mathcal{H})$ such that 
	$(B^\dagger  -A^\dagger) X=B^*-A^*$. In a similar way it can be proved that there exists an  invertible operator $Y\in\mathcal L(\mathcal{H})$ such that $((B^*)^\dagger  -(A^*)^\dagger) Y=B-A$.
	
	Conversely, if $\mathcal{R}(B-A)$ is closed and there exist invertible operators $X,Y\in\mathcal L(\mathcal{H})$ such that $(B^\dagger  -A^\dagger) X=B^*-A^*$ and $((B^*)^\dagger  -(A^*)^\dagger) Y=B-A$ then $\mathcal{R}((B-A)^\dagger)=\mathcal{R}(B^*-A^*)=\mathcal{R}((B^\dagger  -A^\dagger) X)=\mathcal{R}(B^\dagger  -A^\dagger)$ and $\mathcal{R}(((B  -A)^*)^\dagger)=\mathcal{R}(B  -A) = \mathcal{R}(((B^*)^\dagger  -(A^*)^\dagger) Y)= \mathcal{R}((B^*)^\dagger  -(A^*)^\dagger)$. Then the assertion follows.
\end{proof}

\section{Moore-Penrose inverse and operator orders}\label{seccion-monoticidad}

In this section we analyze the behavior of the Moore-Penrose inverse of $A$ and $B$ in $\mathcal{L}(\mathcal{H})$ when the condition  ``$A$ is less or equal than $B$'' holds for the different classes of orders studied in this article.

It is known that the Moore-Penrose inverse is antitonic (or decreasing) on certain subsets of  $\mathcal{L}^h(\mathbb{C}^n)$, see for example \cite{MR1048801}.  Next we study this condition in the infinite dimensional context. The following result refers to  the antitonicity property of the Moore-Penrose inverse on $\mathcal{L}^+$.

\begin{teo}\label{lowner-order}
Consider $A,B\in \mathcal{L}^+$  closed range operators. Then any two of the following conditions imply the third condition:
\begin{enumerate}
\item $A\leq B$;
\item $B^\dagger\leq A^\dagger$;
\item $\mathcal{R}(A)\cap\mathcal{N}(B)=\mathcal{R}(B)\cap\mathcal{N}(A)=\{0\}$.
\end{enumerate}

\end{teo}
\begin{proof} 
	First observe that  $T^\dagger= (T^{1/2})^\dagger (T^{1/2})^\dagger$ for  $T\in\mathcal{L}^+$ with closed range. Now, suppose $A\leq B$ and $B^\dagger\leq A^\dagger$. Since $A$ and $B$ have closed ranges, then by Douglas' theorem it holds that $\mathcal{R}(A)=\mathcal{R}(B)$. Hence item 3 follows. 
	
	Suppose that  $A\leq B$ and $\mathcal{R}(B)\cap\mathcal{N}(A)=\{0\}$. By Douglas' theorem we get that the equation $B^{1/2}X=A^{1/2}$ has solution and that $\|(B^{1/2})^\dagger A^{1/2}\|\leq 1$.  Now, observe that $\|A^{1/2}(B^{1/2})^\dagger\|=\|(B^{1/2})^\dagger A^{1/2}\|\leq 1$. On the other hand,  $\mathcal{R}(B)\subseteq \mathcal{R}(A)$. In fact, if $y\in \mathcal{R}(B)$ then $y=y_1+y_2$ where $y_1\in \mathcal{R}(A)$ and $y_2\in \mathcal{R}(B)\cap\mathcal{N}(A)$. Then,  by the hypothesis, it holds that $y-y_1=0$ and  then,  $\mathcal{R}(B)\subseteq \mathcal{R}(A)$. Therefore $\mathcal{R}((B^{1/2})^\dagger)\subseteq\mathcal{R}((A^{1/2})^\dagger)$. Thus, again by Douglas' theorem it holds that there exists $\lambda \geq 0$ such that  $B^\dagger \leq \lambda A^\dagger$. Now consider the closed and convex set  $\Lambda=\{\lambda \geq 0: \  B^\dagger \leq \lambda A^\dagger\}$.   Since  $ \|A^{1/2}(B^{1/2})^\dagger\|=\underset{\lambda\in\Lambda}{inf}{\lambda}$, it follows that $ \|A^{1/2}(B^{1/2})^\dagger\|\in\Lambda$. Hence,    $B^\dagger\leq A^\dagger$. The proof that items 1 and 3 imply item 2 is similar to the above one.
\end{proof}

\begin{cor}
Let $A,B\in \mathcal{L}^+$ be closed range operators such that $\mathcal{R}(A)=\mathcal{R}(B)$. Then $A\leq B$ if and only if $B^\dagger\leq A^\dagger$.
\end{cor}

\begin{rem}
	Theorem \ref{lowner-order} is not true  for Hermitian operators, in general.  In fact, consider $\mathcal{H}=\mathbb{R}^2$ and let $A=\left(\begin{matrix}
	-1 & 0  \\
	0 & 1 
	\end{matrix}\right)\in \mathcal{L}^h$ and $B=\left(\begin{matrix}
	1 & 0  \\
	0 & 2 
	\end{matrix}\right)\in\mathcal{L}^+$. It holds that $A\leq B$ and $\mathcal{R}(A)\cap\mathcal{N}(B)=\mathcal{R}(B)\cap\mathcal{N}(A)=\{0\}$. However, $A^\dagger$ and $B^\dagger$ are not related with the order $\leq$.
\end{rem}

\medskip

The next result can be found in \cite{arias2021partial}. We include an alternative proof of it. 
\begin{pro} \label{estrella} 
	Consider $A,B\in\mathcal{L}(\mathcal{H})$ closed range operators. Then,  $A\overset{*}{\leq }B$ if and only if $A^\dagger\overset{*}{\leq }B^\dagger$. 
\end{pro}

\begin{proof}
	If $A\overset{*}{\leq }B$ then  $\mathcal{R}(B)=\mathcal{R}(A)\oplus\mathcal{R}(B-A)$ and $\mathcal{R}(B^*)=\mathcal{R}(A^*)\oplus\mathcal{R}(B^*-A^*)$. By \cite[Corollary 3.5]{MR3682701}, it holds that $R(B-A)$ is closed. Also,  $B^\dagger-A^\dagger =(B-A)^\dagger$, see \cite[Corollary 4.10]{MR3682701}. Therefore, $\mathcal{R}(B^\dagger-A^\dagger)=\mathcal{R}(B^*-A^*)$ and $\mathcal{R}((B^*)^\dagger- (A^*)^\dagger )=\mathcal{R}(B-A)$. Hence,  $\mathcal{R}(B^\dagger)=\mathcal{R}(B^*)=\mathcal{R}(A^*)\oplus \mathcal{R}(B^*-A^*)=\mathcal{R}(A^\dagger)\oplus\mathcal{R}(B^\dagger-A^\dagger)$ and $\mathcal{R}((B^\dagger)^*)=\mathcal{R}(B)=\mathcal{R}(A)\oplus\mathcal{R}((B^*)^\dagger-(A^*)^\dagger)$, i.e. $A^\dagger\overset{*}{\leq }B^\dagger$. The converse is similar.
\end{proof}

\begin{cor} Consider  $A,B\in\mathcal{L}(\mathcal{H})$ such that  $A$ is an EP operator and $B$ admits group inverse. Then
	$A\overset{\sharp} \leq B$ if and only if $A^\dagger\overset{\sharp} \leq B^\dagger$.
\end{cor}

\begin{proof}
It follows from the above proposition and Remark \ref{relationships between orders}.
\end{proof}

Now we prove that the Moore-Penrose map is increasing  on the set of Hermitian operators if we consider $A$ and $B$ related by both the  L\"owner and the star orders. 

\begin{teo}\label{lowner y estrella}
	Consider $A,B\in\mathcal{L}^h$  closed range operators. Then, $A \leq B$ and $A \overset{*}{\leq} B$ if and only if $A^\dagger \leq B^\dagger$ and  $A^\dagger \overset{*}{\leq} B^\dagger$.  
\end{teo}
\begin{proof}
	Suppose $A  \leq  B$ and $A \overset{*}{\leq} B$. Then  $\mathcal{R}(B)=\mathcal{R}(A)\oplus\mathcal{R}(B-A)$ and, by \cite[Theorem 2.3]{fillmore1971operator} it holds that $\mathcal{R}(B-A)$ is closed.
	Moreover,  by  \cite[Corollary 4.10]{MR3682701}, it holds that $B^\dagger-A^\dagger =(B-A)^\dagger \geq 0$. Hence $A^\dagger \leq B^\dagger$.
	Also, by Proposition \ref{estrella}, it holds that $A^\dagger \overset{*}{\leq} B^\dagger$.  The converse follows similarly.
\end{proof}

\begin{rem}\label{lowymenos}
	Note that if we consider the orders $\leq$ and $\overset{-}\leq$ in Theorem  \ref{lowner y estrella} then the result is false, in general. In fact, consider $\mathcal{H}=\mathbb{R}^2$ and the operators $A,B\in\mathcal{L}^h$ defined by  $A=\left(\begin{matrix}
	1 & 1  \\
	1 & 1  \\
	\end{matrix}\right)$ and $B=\left(\begin{matrix}
	1 & 1  \\
	1 & 4  \\
	\end{matrix}\right)$. Then it holds that $A\leq B$ and  $A\overset{-}{\leq}B$.  Now, note that $A^\dagger=\left(\begin{matrix}
	1/4 & 1/4  \\
	1/4 & 1/4  \\
	\end{matrix}\right)$, $B^\dagger=\left(\begin{matrix}
	4/3 & -1/3  \\
	-1/3 & 1/3  \\
	\end{matrix}\right)$ and  $A^\dagger\overset{-}{\not\leq}B^\dagger$ because  $\mathcal{R}(B^\dagger - A^\dagger)\cap\mathcal{R}(A^\dagger)\neq \{0\}$.
\end{rem}

\begin{rem}
	If  the orders $\leq$ and $\overset{\diamond}\leq$ are considered in Theorem  \ref{lowner y estrella} then the result is false, in general. In fact, consider $\mathcal{H}=\mathbb{R}^3$ and $A, B\in\mathcal{L}^h$ defined by $A=\left(\begin{matrix}
	1 & 1 & 0  \\
	1 & 1 & 0  \\
	0 & 0 & 0
	\end{matrix}\right)$ and $B=\left(\begin{matrix}
	1 & 1 & 0  \\
	1 & 1 & 1  \\
	0 & 1 & 1
	\end{matrix}\right)$. It is easy to check that  $A\leq B$ and $A\overset{\diamond}{\leq}B$. Now, $A^\dagger=\left(\begin{matrix}
	1/4 & 1/4  & 0 \\
	1/4 & 1/4 & 0  \\
	0 & 0 & 0
	\end{matrix}\right)$ and $B^\dagger=\left(\begin{matrix}
	0 & 1 & -1  \\
	1 & -1 & 1  \\
	-1 & 1 & 0
	\end{matrix}\right)$. But $B^\dagger-A^\dagger=\left(\begin{matrix}
	-1/4 & 3/4 & -1  \\
	3/4 & -5/4  & 1 \\
	-1 & 1 &0
	\end{matrix}\right)\not\geq 0$.  Then $A^\dagger \not\leq B^\dagger$.
\end{rem}

  As a consequence of Proposition \ref{estrella}, the following result gives a new charaterization of the star order in terms of operator ranges.
  Recall that given $A, B\in\mathcal{L}(\mathcal{H})$ with closed ranges,  
  $
  A\overset{\diamond}\leq B \textrm{  if and only if  } \mathcal{R}(B)=\mathcal{R}(A) \dot +\mathcal{R}((B^\dagger-A^\dagger)^*)  \textrm{  and } \mathcal{R}(B^*)=\mathcal{R}(A^*) \dot +\mathcal{R}(B^\dagger-A^\dagger),
  $ see Proposition \ref{caracterizaciones ordenes}.
 The next result states that the last sums are orthogonal if and only if the operators are related with the star order.

\begin{pro}
	Consider $A,B\in\mathcal{L}(\mathcal{H})$ with closed ranges. Then 
 $A\overset{*}\leq B$  if and only if  $\mathcal{R}(B)=\mathcal{R}(A) \oplus \mathcal{R}((B^\dagger-A^\dagger)^*)$ and $\mathcal{R}(B^*)=\mathcal{R}(A^*) \oplus \mathcal{R}(B^\dagger-A^\dagger)$.
\end{pro}

\begin{proof}
It follows from Proposition \ref{caracterizaciones ordenes} and Proposition \ref{estrella}. 
\end{proof}

\begin{pro}\label{minusydagger} Let  $A, B\in\mathcal{L}(\mathcal{H})$ with closed ranges such that  $\mathcal{R}((B-A)^\dagger)=\mathcal{R}(B^\dagger-A^\dagger)$ and $\mathcal{R}((B^*-A^*)^\dagger)=\mathcal{R}((B^*)^\dagger-(A^*)^\dagger)$. Then
	$A\overset{-}\leq B$ 
	if and only if $A^\dagger \overset{-}\leq B^\dagger$.
\end{pro}
\begin{proof}
	It follows from Theorem \ref{minus equiv diamond} and \cite[Corollary 4.5]{arias2021partial}.
\end{proof}

\begin{cor}\label{diamondydagger} Let  $A, B\in\mathcal{L}(\mathcal{H})$ with closed ranges such that  $\mathcal{R}((B-A)^\dagger)=\mathcal{R}(B^\dagger-A^\dagger)$ and $\mathcal{R}((B^*-A^*)^\dagger)=\mathcal{R}((B^*)^\dagger-(A^*)^\dagger)$. Then
	$A\overset{\diamond}\leq B$ 
	if and only if $A^\dagger \overset{\diamond}\leq B^\dagger$.
\end{cor}

\begin{proof}
It follows from the above proposition and Theorem \ref{minus equiv diamond}.
\end{proof}

\begin{rem}\label{monoticidadparaelmenos}
Proposition \ref{minusydagger} is false, in general, if the condition  $\mathcal{R}((B-A)^\dagger)=\mathcal{R}(B^\dagger-A^\dagger)$ and $\mathcal{R}((B^*-A^*)^\dagger)=\mathcal{R}((B^*)^\dagger-(A^*)^\dagger)$ is not required. In fact, consider $\mathcal{H}=\mathbb{R}^2$ and the operators $A,B\in\mathcal{L}(\mathcal{H})$ defined as in Remark \ref{lowymenos}. Then it holds that $A\overset{-}{\leq}B$ but $A^\dagger\not\overset{-}{\leq} B^\dagger$ because $\mathcal{R}(B^\dagger - A^\dagger)\cap\mathcal{R}(A^\dagger)\neq \{0\}$. 
\end{rem}
\medskip

\section{$B^\dagger A^\dagger$ as a generalized inverse of $AB$}\label{seccion-ROL}

Given $A, B\in\mathcal{L}(\mathcal{H})$ with closed ranges such that $AB$ has closed range,  the  reverse order law  property for the Moore-Penrose inverse, $(AB)^\dagger=B^\dagger A^\dagger$, has been studied for a long time. Several characterizations  of the reverse order law were given in terms of range inclusion conditions and commutativity between operators,  see \cite{MR317088,MR2567299,MR651705}  among other articles.  The following is a classical result due to Greville \cite{greville1966note} that was extended to bounded linear operators in Hilbert spaces in \cite{MR317088, MR651705}. 
Before that  remember that $(AB)^\dagger\in\mathcal{L}(\mathcal{H})$ if and only if $\mathcal{R}(AB)$ is closed. Then it is worth recalling that given $A,B\in\mathcal{L}(\mathcal{H})$ with closed ranges, $\mathcal{R}(AB)$ is closed  if and only if $\mathcal{R}(B)+\mathcal{N}(A)$ is closed;  see \cite{MR326424} or \cite[Theorem 4.1]{MR1340886}. 

\begin{pro}\label{dagger de AB con rg iguales} Let $A,B\in\mathcal{L}(\mathcal H)$ with closed ranges such that  $AB$ has closed range.  Then  $(AB)^\dagger=B^\dagger A^\dagger$ if and only if $\mathcal{R}(A^*AB)\subseteq \mathcal{R}(B)$ and $\mathcal{R}(BB^*A^*)\subseteq \mathcal{R}(A^*)$.
\end{pro}

In the literature the reverse order law is also study   by means of the star, minus and sharp orders in the finite dimensional context, see for example \cite{malik2007matrix, MR2653538}. In this section we study the reverse order law for closed range operators in the infinite dimensional context in terms of  order and  range inclusion conditions. Some of the results that we expose are extensions of known results in the finite dimensional case.   

The following  characterization of the Moore-Penrose inverse will be useful in what follows, its proof can be found in \cite[Theorem 3.1]{arias2008generalized}:

\begin{lem}\label{MP-reduced solution}
Let $T\in\mathcal{L(\mathcal H)}$ be a closed range operator. Then, $T^\dagger$ is the Douglas'  reduced solution of the equation $TX=P_T$.
\end{lem}

  There is a natural connection between EP operators and the reverse order law. Indeed, given EP operators with the same range, then the reverse order law holds, see \cite[Theorem 5]{djordjevic2001products}. We refer also to  \cite{baskett1969theorems} for finite dimensional results involving EP operators and the reverse order law. The next result provides a sufficient condition to obtain the reverse order law  for EP operators, in the infinite dimensional context.

\begin{teo}\label{AEP-reverseorderlaw}
	Consider $A, B\in\mathcal{L}(\mathcal{H})$ such that $A, B, AB$ are EP operators and  $\mathcal R(A)=\mathcal R(AB)$.
	
	   Then $(AB)^\dagger=B^\dagger A ^\dagger$. 
\end{teo}
\begin{proof}
	We will prove that 
	$B^\dagger A^\dagger$ is the  Douglas' reduced solution of the operator equation $ABX=P_{AB}$. 
	In fact, observe that $ABB^\dagger A^\dagger=AP_BA^\dagger=AA^\dagger=P_A=P_{AB}$, where the second equality follows because $\mathcal R(A^\dagger)=\mathcal R(A^*)=\mathcal R(A)=\mathcal{R}(B^*A^*)\subseteq\mathcal R(B^*)=\mathcal{R}(B)$. Furthermore, $\mathcal{R}(B^\dagger A^\dagger)= \mathcal{R}(B^\dagger A^*)=\mathcal{R}(B^\dagger A)=\mathcal{R}(B^\dagger BA)=\mathcal{R}(P_{B^*}A)=\mathcal{R}(P_{B}A)=\mathcal{R}(A)=\mathcal{R}(B^*A^*)=\mathcal{N}(AB)^\bot$. Therefore $B^\dagger A^\dagger$ is the Douglas' reduced solution of $ABX=P_{AB}$ and so, by  Lemma \ref{MP-reduced solution}, it follows that $(AB)^\dagger =B^\dagger A^\dagger$. 
\end{proof}

\begin{rem}
Note that if $A,B\in\mathcal{L}(\mathcal{H})$ are EP operators with $\mathcal{R}(A)=\mathcal{R}(B)$ then $AB$ is also an EP operator and $\mathcal{R}(A)=\mathcal{R}(AB)$. However, if $A, B$ and $AB$ are EP operators such that  $\mathcal{R}(A)=\mathcal{R}(AB)$ then $\mathcal{R}(A)$ is not necessarily equal to $\mathcal{R}(B)$. In fact, in $\mathcal{L}(\mathbb{R}^2)$ consider $A=\left(\begin{matrix}
	1 &  0  \\
	0 &  0  
	\end{matrix}\right)$ and $B= \left(\begin{matrix}
	1 &  0  \\
	0 &  1  
	\end{matrix}\right)$. It holds that $A, B$ and $AB$ are EP operators such that $\mathcal{R}(A)=\mathcal{R}(AB)$ but $\mathcal{R}(A)\neq \mathcal{R}(B)$. This remark shows that the hypothesis of Proposition \ref{AEP-reverseorderlaw} is weaker than  the hypothesis of \cite[Theorem 5]{djordjevic2001products}.
\end{rem}

The next result is a consequence of Proposition \ref{dagger de AB con rg iguales}. We include an alternative proof.

\begin{pro}\label{conmutan entonces reverse oreder law}
	Let $A,B\in\mathcal{L}^h$ be closed range operators such that $AB=BA$. Then $(AB)^\dagger=B^\dagger A^\dagger$.
\end{pro}
\begin{proof} 
 First observe that since $AB=BA$, then $AP_B=P_BA$, $P_AB=BP_A$ and  $P_AP_B=P_BP_A$ (it follows as a consequence of the functional calculus for selfadjoint operators). Therefore, it also holds that $P_AP_B$ is an orthogonal projection.
  Then, by \cite[Proposition 2.1]{MR651705}, it follows that $\mathcal{R}(AB)$ is closed.
 
 Now, we will prove that  $(AB)^\dagger=B^\dagger A^\dagger$. Note that 
 $ABB^\dagger A^\dagger A B=AP_BP_AB=AP_AP_BB=AB$. Similarly, $B^\dagger A^\dagger AB B^\dagger A^\dagger=B^\dagger A^\dagger$. Moreover, note that $AB B^\dagger A^\dagger=AP_B  A^\dagger=P_BP_A$ is selfadjoint because $P_BP_A$ is an orthogonal projetions. In a similar way, it follows that $B^\dagger A^\dagger AB$ is selfadjoint. Hence $(AB)^\dagger=B^\dagger A^\dagger$.
\end{proof}

\begin{pro} \label{ROLyantitonicity}
	Let $A,B\in\mathcal{L}^+$ with closed ranges. If $A\leq B$ and $B^\dagger \leq A^\dagger$ then $ (AB)^\dagger=B^\dagger A^\dagger$ and  $(BA)^\dagger=A^\dagger B^\dagger$. 
\end{pro}
\begin{proof}
	Suposse  $A\leq B$ and $B^\dagger\leq A^\dagger$. Then, as observed in the proof of Theorem \ref{lowner-order}, it holds that $\mathcal{R}(A)=\mathcal{R}(B)$. So that $\mathcal{R}(B)+\mathcal{N}(A)=\mathcal{R}(B)\oplus\mathcal{N}(B)$ is a closed subspace and then $AB$ has closed range. Hence, by Proposition \ref{dagger de AB con rg iguales} it follows that  $B^\dagger A^\dagger=(AB)^\dagger$ and $A^\dagger B^\dagger=(BA)^\dagger$
\end{proof}

In the next proposition we give a necessary and sufficient condition in terms of the star order  for the validity of  the reverse order law. A similar result is proved in \cite[Theorem 2.1]{MR2653538} for matrices. Here we present a different proof in the context of Hilbert space operators. 
\begin{pro}\label{ROL-AEP}
	Consider $A, B\in\mathcal{L}(\mathcal{H})$ with closed range such that $A$ is an   EP operator. Then,  $A\overset{*}{\leq }B$ if and only if $(AB)^\dagger=B^\dagger A ^\dagger$ and $A=P_AB$. 
\end{pro}
\begin{proof}
	Suppose $A\overset{*}{\leq }B$. Then $\mathcal{R}(A)\subseteq \mathcal{R}(B)$. Now, since $A$ is an EP operator it holds that $\mathcal{R}(A^*)\subseteq \mathcal{R}(B)$ and so $\mathcal{R}(B)^\bot\subseteq \mathcal{N}(A)$. So that, $\mathcal{H}=\mathcal{R}(B)+\mathcal{R}(B)^\bot\subseteq \mathcal{R}(B)+\mathcal{N}(A)$. In consequence, $ \mathcal{R}(B)+\mathcal{N}(A)$ is a closed subspace  and then $\mathcal{R}(AB)$ is closed. Now, observe that $\mathcal{R}(A^*AB) \subseteq \mathcal{R}(A^*)\subseteq \mathcal{R}(B)$.  On the other hand, since $A\overset{*}{\leq }B$ then  it  holds that $A^*A=A^*B$ and $AA^*=BA^*$. So, by the above equalities and by the fact that $\mathcal{R}(A)=\mathcal{R}(A^*)$,  we get that $\mathcal R(BB^*A^*)=\mathcal R(BB^*A)=\mathcal R(BA^*A)=\mathcal{R}(AA^*A)\subseteq \mathcal R(A)=\mathcal R(A^*)$. Hence, by Proposition \ref{dagger de AB con rg iguales}, it holds that $(AB)^\dagger=B^\dagger A ^\dagger$. Equality $A=P_AB$ follows from Proposition \ref{rem-orden-estrella}. Conversely, since $A=P_AB$ then $A^*A=A^*B$. Now,  as $(AB)^\dagger=B^\dagger A^\dagger$ and $A$ is an EP operator  then, by Proposition \ref{dagger de AB con rg iguales},  $\mathcal{R}(BB^*P_A)=\mathcal{R}(BB^*A)\subseteq \mathcal{R}(A)$. Then we get that $AA^*=P_ABB^*P_A=BB^*P_A=BA^*$. Therefore, $A\overset{*}{\leq }B$ and the assertion follows.
\end{proof}

The following result can be found in \cite[Theorem 3.2]{malik2007matrix} for  finite dimensional spaces. Even thought the proof presented by S. Malik can be extended to operators defined on infinite Hilbert spaces, we will prove the next result in a different way.

\begin{pro}\label{ROL-estrella}
	Let $A\in\mathcal{L}(\mathcal{H})$  and $B\in\mathcal{L}^h$ with closed range.   If $A\overset{*}{\leq }B$ then $(AB)^\dagger=B^\dagger A ^\dagger$. 
\end{pro}

\begin{proof}
Suppose $A\overset{*}{\leq }B$ and $\mathcal{R}(B)$ is closed. As a consequence of \cite[Corollary 3.5]{MR3682701}, it follows that $\mathcal R(A)$ is closed.
	Also note that $\mathcal R(AB)$ is closed. In fact, since $A\overset{*}{\leq }B$ and $B=B^*$ then $\mathcal R(A^*)\subseteq \mathcal R(B)$, so that $\mathcal N(B) \subseteq \mathcal N(A)$ and $\mathcal R(B)+\mathcal N(A)$   is closed since $\mathcal R(B)+\mathcal N(A)=\mathcal N(B)^\perp+\mathcal N(A)\supseteq \mathcal N(B)^\perp+\mathcal N(B)$. Then 
	$\mathcal{R}(AB)$ is closed. 
	Now, we will show  that $B^\dagger A ^\dagger$ is the Douglas' reduced solution of the equation $ABX=P_{AB}$. It is not difficult to see that $\mathcal N(A^*)=\mathcal N(BA^*)$ because $\mathcal R(A^*)\cap \mathcal N(B)=\{0\}.$ Then $\mathcal R (AB)=\mathcal N(BA^*)^\perp=\mathcal R(A)$. Now, observe that $ABB^\dagger A^\dagger =AP_B A^\dagger =AA^\dagger =P_A=P_{AB}$.  Finally,  it holds that $\mathcal R(B^\dagger A ^\dagger)\subseteq \mathcal N(AB)^\perp$. In fact,  $\mathcal R(B^\dagger A ^\dagger)=\mathcal R(B^\dagger A ^*)=\mathcal R((A^*)^\dagger A ^*)=\mathcal{R}(A) =\mathcal{R}(AA^\dagger)= \mathcal{R}(BA^\dagger)=\mathcal{R}(BA^*)= \mathcal{N}(AB)^\perp$; where the second and the fifth equality follow from Proposition \ref{rem-orden-estrella}. Hence $B^\dagger A^\dagger$ is the  Douglas' reduced solution of the equation $ABX=P_{AB}$, and, by Lemma \ref{MP-reduced solution}, it follows that $(AB)^\dagger =B^\dagger A^\dagger$. 
\end{proof}

\begin{rem}
In \cite{malik2007matrix} it is shown that the condition $B\in\mathcal{L}^h$ in the above result can not be replaced by the weaker hypothesis of $B$ being an EP operator.
\end{rem}

\begin{cor}
Consider $A\in\mathcal{L}(\mathcal{H})$  an EP operator and $B\in\mathcal{L}^h$ a closed range operator.   If $A\overset{\sharp}{\leq }B$ then $(AB)^\dagger=B^\dagger A ^\dagger$. 
\end{cor}

\begin{proof}
It follows from the above proposition and Proposition  \ref{relationships between orders}.
\end{proof}

\begin{rem}
Proposition \ref{ROL-estrella} does not hold if the hypothesis $A\overset{*}{\leq }B$ is replaced by $A\overset{-}{\leq }B$. In fact, consider $\mathcal{H}=\mathbb{R}^2$ and $A,B\in\mathcal{L}(\mathcal{H})$ defined by $A=\left(\begin{matrix}
	1 &  0  \\
	1 &  0  
	\end{matrix}\right)$ and $B=\left(\begin{matrix}
	1 &  1  \\
	1 &  0  
	\end{matrix}\right)$. Then it holds that $A\overset{-}{\leq} B$ but $A\overset{*}{\not\leq} B$. Now, $A^\dagger=\left(\begin{matrix}
	1/2 &  1/2  \\
	0 &  0  
	\end{matrix}\right)$, $B^\dagger=\left(\begin{matrix}
	0 &  1  \\
	1 &  -1  
	\end{matrix}\right)$ and $(AB)^\dagger=\left(\begin{matrix}
	1/4 &  1/4  \\
	1/4 &  1/4  
	\end{matrix}\right)$. Then, $(AB)^\dagger\neq B^\dagger A^\dagger$.
\end{rem}

\begin{rem}
Proposition \ref{ROL-estrella} is not true if the hypothesis $A\overset{*}\leq B$ is replaced by $A\overset{\diamond}\leq B$. In fact, consider  $\mathcal{H}=\mathbb{R}^3$ and take $A,B\in\mathcal{L}^h$ defined by  $A=\left(\begin{matrix}
	1 & 1 & 0  \\
	1 & 1 & 0  \\
	0 & 0 & 0
	\end{matrix}\right)$ and  $B=\left(\begin{matrix}
	1 & 1 & 0  \\
	1 & 1 & 1  \\
	0 & 1 & 1
	\end{matrix}\right)$. Then, it holds that $A\overset{\diamond}\leq B$. Now, $A^\dagger=\left(\begin{matrix}
	1/4 & 1/4  & 0 \\
	1/4 & 1/4 & 0  \\
	0 & 0 & 0
	\end{matrix}\right)$,  $B^\dagger=\left(\begin{matrix}
	0 & 1 & -1  \\
	1 & -1 & 1  \\
	-1 & 1 & 0
	\end{matrix}\right)$ and  $(AB)^\dagger=\frac{1}{18}\left(\begin{matrix}
	2& 2 & 0  \\
	2& 2 & 0  \\
	1 & 1 & 0
	\end{matrix}\right)$. Then,  $(AB)^\dagger\neq B^\dagger A^\dagger$.
\end{rem}

\begin{pro}\label{ROL-pre-orden}
Let $A\in\mathcal{L}^h$  and $B\in\mathcal{L}(\mathcal{H})$ with closed ranges such that $B^\dagger AB$ is an $EP$ operator and $A\overset{s}{\leq} B$. Then $(AB)^\dagger=B^\dagger A^\dagger$.
\end{pro}

\begin{proof}
Let us see that $B^\dagger A^\dagger $ is the Douglas' reduced solution of $ABX=P_{AB}$. Note that since  $A\in\mathcal{L}^h$ and  $A\overset{s}{\leq} B$ then $\mathcal{R}(AB)$ is closed and $\mathcal{R}(AB)=\mathcal{R}(A)$. Then consider the equation $ABX=P_A$. Observe that, $ABB^\dagger A^\dagger=AP_BA^\dagger =P_A$. In addition, $\mathcal{R}(B^\dagger A^\dagger)=\mathcal{R}(B^\dagger A)=\mathcal{R}(B^\dagger A B)=\mathcal{R}(B^*A(B^*)^\dagger)\subseteq \mathcal{R}(B^*A)\subseteq \mathcal{N}(AB)^\bot$. Then the assertion follows from \cite[Theorem 3.1]{arias2008generalized}.
\end{proof}

\begin{rem}
By Proposition \ref{ROL-estrella} and  \cite[Corollary 4.10]{MR3682701} it holds that  $A\overset{*}{\leq} B$ implies the reverse order law and the dagger sustractivity property. Then  a natural question emerges:  Is there any relationship between these two properties? The following examples show that neither condition implies the other. In fact, 
in $\mathcal{L}(\mathbb{R}^2)$ consider $A=\left(\begin{matrix}
	1 &  0  \\
	0 &  0  
	\end{matrix}\right)$ and $B=\left(\begin{matrix}
	1 &  1  \\
	1 &  -1  
	\end{matrix}\right)$. Then it can be checked that  $(AB)^\dagger=B^\dagger A^\dagger$ but $(B-A)^\dagger\neq B^\dagger- A^\dagger$. On the other hand, if we consider $A=\left(\begin{matrix}
	1 &  0  \\
	1 &  0  
	\end{matrix}\right)$ and $B=\left(\begin{matrix}
	1 &  1  \\
	1 &  -1  
	\end{matrix}\right)$ then $(B-A)^\dagger =B^\dagger - A^\dagger$ but $(AB)^\dagger \neq B^\dagger A^\dagger$.
\end{rem}

If $A, B\in \mathcal{L}(\mathcal{H})$ are  EP operators such that $\mathcal R(A)\subseteq\mathcal R(B)$ then it is not difficult to see that  $B^\dagger A^\dagger\in AB[1,2,3]$.
 In the following two results we show that $B^\dagger A^\dagger$ can be  considered as a weighted generalized inverse of $AB$ for certain weights.

\begin{teo}\label{BdaggerAdagger-wgi}
Suppose that  $A, B\in \mathcal{L}(\mathcal{H})$  are  EP operators such that $\mathcal{R}(A)\subseteq \mathcal{R}(B)$. Then, $B^\dagger A^\dagger$ is an $I,M$-weighted generalized inverse of $AB$ for every $M\in \mathcal{L}^+$ such that $M=M_1+M_2$, where $M_1,M_2\in \mathcal{L}^+$, $\mathcal{R}(M_1)\subseteq \mathcal{R}(B^*A)$ and $\mathcal{R}(M_2)\subseteq \mathcal{N}( AB^\dagger )$.
\end{teo}

\begin{proof}
	It is not difficult to see that in this case  $B^\dagger A^\dagger\in AB[1,2,3]$. Observe that $B^\dagger A^\dagger AB=B^\dagger P_AB$ is an idempotent. In addition, $\mathcal{R}((B^\dagger P_AB)^*)=\mathcal{N}(B^\dagger P_AB)^\bot=\mathcal{N}(P_AB)^\bot=\mathcal{R}(B^*P_A)=\mathcal{R}(B^*A)$ (note that $\mathcal{R}((B^*P_A)$ is closed because $\mathcal R(A)+\mathcal N(B^*)$ is closed). Also, it holds that  $\mathcal{N}((B^\dagger P_AB)^*)=\mathcal{N}(B^* P_A(B^\dagger)^*)=\mathcal{N}( P_A(B^\dagger)^*)=\mathcal{R}( (B^\dagger P_A))^\bot=\mathcal{R}( (B^\dagger A))^\bot=\mathcal{R}( (B^\dagger A^*))^\bot=\mathcal{N}( A(B^*)^\dagger )$.
	Then, by \cite[Theorem 5.1]{MR3119120}, it follows that $MBP_AB^\dagger$ is Hermitian for every $M\in\mathcal{L}^+$ such that $M=M_1+M_2$, where $M_1,M_2\in \mathcal{L}^+$, $\mathcal{R}(M_1)\subseteq \mathcal{R}(B^*A)$ and $\mathcal{R}(M_2)\subseteq \mathcal{N}( A(B^*)^\dagger )$. Therefore, $B^\dagger A^\dagger$ is an $I,M$-weighted generalized inverse of $AB$.
\end{proof}

\begin{pro}
	Let $A \in \mathcal L(\mathcal H)$ an EP operator and $ B\in \mathcal{L}^h$ with closed range such that $\mathcal{R}(A)\subseteq \mathcal{R}(B)$. Then $B^\dagger A^\dagger$ is an $I,B^2$-weighted generalized inverse of $AB$. 
\end{pro}

\begin{proof}
	It holds that $ B^\dagger A^\dagger\in AB[1,2,3]$ and note that $B^2\in\mathcal{L}^+$. Then, it only remains to show that $B^2B^\dagger P_AB \in\mathcal{L}^h$. In fact, $B^2B^\dagger P_AB= BP_B P_AB=BP_AB\in\mathcal{L}^h$. Hence, $B^\dagger A^\dagger$ is an $I,B^2$-weighted generalized inverse of $AB$.  
\end{proof}

Recall that if $A$ and $B$ are Hermitian operators with closed ranges such that $AB=BA$ then $B^\dagger A^\dagger$ is the Moore-Penrose inverse of $AB$, see Proposition \ref{conmutan entonces reverse oreder law}. As a consequence of the following lemma, we will state  that $B^\dagger A^\dagger$ is also a weighted inverse of $AB$, if $A$ or $B$ are positive operators.

\begin{lem}
	Consider $T\in \mathcal L(\mathcal H)$ and $ A\in \mathcal L^+$. Then $T^ \dagger $ is an $A$-inverse of $T$ if and only if $A\mathcal R(T)\subseteq \mathcal R(T)$.
\end{lem}

\begin{proof}
	Observe that $T^ \dagger $ is an $A$-inverse of $T$ if and only if $T^*ATT^\dagger=T^*A$, see Proposition \ref{caracterizacion de M-inversas}. Equivalently, $T^*AP_T=T^*A$, or $\mathcal{R}(T)^\perp \subseteq \mathcal N(T^*A)=\mathcal R(AT)^\perp$. Therefore,  the assertion follows. 
\end{proof}

\begin{pro}\label{A inverse of AB}
	Let $A, B\in \mathcal{L}^h$ with closed ranges such $AB=BA$. Then the following assertions hold:
	\begin{enumerate}
		\item If $A\in \mathcal{L}^+$ then $B^\dagger A^\dagger$ is an $A$-inverse of $AB$.
		\item If $B\in \mathcal{L}^+$ then $B^\dagger A^\dagger$ is an $B$-inverse of $AB$.
	\end{enumerate} 
	Moreover, in any of the above cases, the set of $A$-inverses of $AB$ or $B$-inverses of $AB$ is given by
	$$
	\{B^\dagger A^\dagger + Z: Z\in\mathcal{L}(\mathcal{H}) \ and \ \mathcal{R}(Z)\subseteq \mathcal{N}(AB)\}.
	$$
\end{pro}

\begin{proof}
	1. Consider $A\in \mathcal{L}^+$. By the above lemma and  by the fact that $A$ commutes with $B$, it holds that $B^\dagger A^\dagger$ is an $A$-inverse of $AB$. Moreover, by \cite[Proposition 5.9]{MR3085500} all $A$-inverses of $AB$ are the solutions of the operator equation $BA^3BX=BA^2$. Therefore, since $\mathcal{N}(BA^3B)= \mathcal{N}(A^{3}B^2)=\mathcal{N}(AB)$ then all $A$-inverses of $AB$ are  $B^\dagger A^\dagger + Z$ with $Z\in\mathcal{L}(\mathcal{H})$ such that $\mathcal{R}(Z)\subseteq \mathcal{N}(AB)$. 
	
	2. The proof is similar to the proof of item 1. 		
\end{proof}

 We finish by showing that $B^\dagger A^\dagger$ is a solution of certain operator weighted least squares problem related to $AB$ when the weighted norm comes from  a Schatten $p$-norm.  Let $1\leq p <\infty$,  remember that if $T\in\mathcal{L}(\mathcal{H})$ is a compact operator and the sequence $\{\lambda_k(T)\}_{k\in\mathbb{N}}$ denotes the eigenvalues of $|T|$, where each eigenvalue is repeated according its multiplicity, then it is said that $T$ belongs to the Schatten $p$-class $S_p(\mathcal{H})$ if $\sum
\limits_{k\geq 1} \lambda_k^p(T)<\infty$. In addition, if $T\in S_p(\mathcal{H})$ then its $p$-Schatten norm is $\|T\|_p=(\sum
\limits_{k\geq 1} \lambda_k^p(T))^{1/p}$. 

In \cite{ConGirMae17} different relationships between $M$-inverses and minimization problems that involve weighted $p$-Schatten norms were considered. More precisely, in \cite{ConGirMae17}, it is   studied the existence of the minimum of the set $\{\|W^{1/2}(TX-S)\|_p: X\in\mathcal{L}(\mathcal{H})\}$  when $W\in\mathcal{L}^+$ is such that $W^{1/2}\in S_p$ for $1\leq p <\infty$ . Also, it is shown that  the  operators that achieve the minimum (when  the minimum  exists) are the $W$-inverses of $T$ in $\mathcal{R}(S)$.  We will need the following lemma that is disseminated in \cite{ConGirMae17}. Its proof follows applying \cite[Poroposition 5.7]{MR3085500} and  \cite[Proposition 3.3]{ConGirMae17}.

\begin{lem}\label{M inversa implica pSchattenM inversa}
Let $W \in \mathcal L^+, W^{1/2}\in S_p$ and $ T\in \mathcal L(\mathcal{H})$. If $X_0\in \mathcal{L}(\mathcal{H})$ is a $W$-inverse of $T$, then 
	$\|TX_0-I\|_{p,W}=
	\underset{X\in \mathcal L(\mathcal H)}{\min}\|TX-I \|_{p,W}.$
\end{lem}

Consider $A, B\in \mathcal L^h$ with closed ranges. Observe that, by Douglas' theorem,  $ABX=P_A$ has a solution if and only if $\mathcal R(AB)=\mathcal R(A)$. If, in addition, $\mathcal R(A)\subseteq \mathcal R(B)$ it holds that $B^\dagger A^\dagger$ is a solution. 

In the next result we analyze the equation $ABX=P_A$ when it  is not necessarily solvable. 

\begin{teo}\label{BdaggerAdagger-pwproblem}
	Consider $A\in \mathcal L^h$ with closed range,   $W \in S_p$ with  $1\leq p<\infty$ such that $W$ is injective, $WA\in \mathcal{L}^+$ and $ B\in \mathcal{L}^h$ with closed range such that $\mathcal{R}(B)\subseteq \mathcal{R}(A)$. Then there exists
	\begin{equation}\label{pW-inversa}
	\underset{Y\in \mathcal L(\mathcal H)}{\min}\|ABY-P_A \|_{p,WA}.
	\end{equation}
	Moreover, $B^\dagger A^\dagger$ is a solution.
\end{teo}
\begin{proof}
	Observe that $
	(AB)^*WA(AB)B^\dagger A^\dagger=BAWAAP_BA^\dagger=BAWAP_A=(AB)^*WA$. Then, by Proposition \ref{caracterizacion de M-inversas},  it follows that $B^\dagger A^\dagger$ is a $WA$-inverse of $AB$. Then, by Lemma \ref{M inversa implica pSchattenM inversa}, $B^\dagger A^\dagger$ solves the problem $
	\underset{Y\in \mathcal L(\mathcal H)}{\min}\|ABY-I \|_{p,WA}.$
	But, $\|ABY-I \|_{p,WA}=\|(WA)^{1/2}(ABY-I)\|_{p}=\| (WA)^{1/2}(ABY-P_A)\|_{p}=\| ABY-P_A\|_{p, WA}.
	$
	So that $B^\dagger A^\dagger$ is a solution for (\ref{pW-inversa}).
\end{proof}

\begin{rem}
Under the hypotheses of Theorem \ref{BdaggerAdagger-pwproblem} it holds that $B^\dagger A^\dagger$ is an $WA$-inverse of $AB$ in $\mathcal{R}(A)$ (see  definition of this class of weighted inverses in \cite{ConGirMae17}).
\end{rem}

\textbf{Aknowlegment.} We would like to thank  Santiago Muro for fruitful conversations that help us improve the article.

\end{document}